\pdfoutput=1 
\documentclass[12pt,twoside]{article}

\usepackage{geometry}
\usepackage[T1]{fontenc}        

\usepackage[utf8x]{inputenc}    
\usepackage{ucs}                
\usepackage[comma,authoryear]{natbib}
\usepackage{amsmath}            
\usepackage{amsfonts}
\usepackage{amssymb}
\usepackage{amsthm} 
\usepackage{thmtools}
\usepackage{ifthen}
\usepackage{color}
\usepackage{graphicx} 
\usepackage{textcomp}            
\usepackage{palatino}            
\usepackage{eulervm}            
\usepackage{varioref}            
\usepackage{fancyhdr}
\usepackage{float}
\usepackage{multirow}
\usepackage{rotating}
\usepackage{subfig}
\usepackage{enumerate}

\newcommand{\R}{{\mathbb{R}}}         
\newcommand{\E}{{\mathbb{E}}}
 \newcommand{\N}{{\mathbb{N}}}         

\newcommand{\bay}{\begin{array}}
\newcommand{\eay}{\end{array}}

\newcommand{\bqa}{\begin{eqnarray*}}
\newcommand{\eqa}{\end{eqnarray*}}

\newcommand{\bee}{\begin{eqnarray*}}
\newcommand{\eee}{\end{eqnarray*}}

\newcommand{\bea}{\begin{eqnarray*}}
\newcommand{\eea}{\end{eqnarray*}}

\newcommand{\bqan}{\begin{eqnarray}}
\newcommand{\eqan}{\end{eqnarray}}

\newcommand{\be}{\begin{eqnarray}}
\newcommand{\ee}{\end{eqnarray}}

\newcommand{\bit}{\begin{itemize}}
\newcommand{\eit}{\end{itemize}}

\newcommand{\ben}{\begin{enumerate}}
\newcommand{\een}{\end{enumerate}}

\newcommand{\beq}{\begin{equation}}
\newcommand{\eeq}{\end{equation}}

\newcommand{\bdes}{\begin{description}}
\newcommand{\edes}{\end{description}}

\newcommand{\btb}{\begin{tabular}}
\newcommand{\etb}{\end{tabular}}

\newcommand{\bcen}{\begin{center}}
\newcommand{\ecen}{\end{center}}

\newcommand{\bmp}{\begin{minipage}}
\newcommand{\emp}{\end{minipage}}

\newcommand{\Var}{\operatorname{{\it Var}}}

\newcommand{\tr}{\operatorname{tr}}

\newcommand{\diag}{\operatorname{\it diag}}

\newcommand{\ve}{\boldsymbol{e}}

\newcommand{\vn}{\boldsymbol{n}}

\newcommand{\vH}{\boldsymbol{H}}
\newcommand{\vI}{\boldsymbol{I}}
\newcommand{\vJ}{\boldsymbol{J}}

\newcommand{\vP}{\boldsymbol{P}}

\newcommand{\vT}{\boldsymbol{T}}

\newcommand{\vV}{\boldsymbol{V}}

\newcommand{\vX}{\boldsymbol{X}}
\newcommand{\vY}{\boldsymbol{Y}}

\newcommand{\vmu}{\boldsymbol{\mu}}

\newcommand{\vsigma}{\boldsymbol{\sigma}}
\newcommand{\vSigma}{\boldsymbol{\Sigma}}

\newcommand{\veins}{{\bf 1}}
\newcommand{\vnull}{{\bf 0}}

\newcommand{\ind}{1\hspace{-0.7ex}1}

\newcommand{\Cdot}{\cdot}

\newcommand{\0}{\mathcal{O}}

\newtheoremstyle{Test1}
  {2 \baselineskip}
  {1.5 \baselineskip}
  {\itshape}
  {-0.0ex}
  {\fontfamily{ppl}\fontseries{l}\fontshape{n}}
  {:}
  {\newline}
   {}

\theoremstyle{Test1}
\newtheorem{theorem}{Theorem}[section]

\newtheorem{Le}[theorem]{Lemma}

\newcommand{\Lan}{\mathcal{O}}
\newcommand{\lan}{ \scriptstyle \mathcal{O}\textstyle}

\makeatletter
\renewenvironment{proof}[1][\proofname]{\par
  \pushQED{\qed}%
  \fontfamily{ppl}\fontseries{m}\fontshape{it} \topsep6\p@\@plus6\p@\relax
  \trivlist
  \item[\hskip\labelsep
        \bfseries
    #1\@addpunct{:}]\ignorespaces
}{%
  \popQED\endtrivlist\@endpefalse
}
\makeatother

\begin{document}
%
%
%
%
%

\title{\Large \bf Manifold Asymptotics of Quadratic-Form-Based Inference in 
Repeated Measures Designs}

\author{Paavo Sattler$^{*}$ \\[1ex] 
}
\maketitle

\begin{abstract}
\noindent
Split-Plot or Repeated Measures Designs with multiple groups occur naturally in sciences. Their analysis is usually based on the classical Repeated Measures ANOVA. Roughly speaking, the latter can 
be shown to be asymptotically valid for large sample sizes $n_i$ 
assuming a fixed number of groups $a$ and time points $d$. However, for high-dimensional settings with $d>n_i$  this argument breaks down and statistical tests are often based on (standardized) quadratic forms. Furthermore analysis of their limit behaviour is usually based on certain assumptions on how $d$ converges to $\infty$ with respect to $n_i$. As this may be hard to argue in practice, we do not want to make such restrictions. Moreover, sometimes also the number of groups $a$ may be 
large compared to $d$ or $n_i$. To also have an impression about the behaviour of (standardized) quadratic forms as test statistic, we analyze their asymptotics under diverse settings on $a$, $d$ and $n_i$. 
In fact, we combine all kinds of combinations, where they diverge or are bounded in a unified framework. Studying the limit distributions in detail, we follow Sattler and Pauly (2018) and propose an approximation to obtain critical values. The resulting test together with their approximation approach is investigated in an extensive simulation study with a focus on the exceptional asymptotic frameworks which are the main focus of this work.

\end{abstract}

\noindent{\bf Keywords:} Multivariate Data, equal covariance matrices, high-dimensions, repeated measures.

\vfill
\vfill

\noindent${}^{*}$ {Dortmund University, Faculty of Statistics, Germany}\\

\newpage

\section{Motivation and Introduction}\label{int}
In many studies, it is possible to conduct and handle a large number of measurements, which makes high-dimensionality an increasingly important topic. In fact, high dimensional repeated measure designs or split-plot designs for multiple groups are the objectives of many analyses in science. This is the case in life science, where test persons were examined multiple times during a study, or in the industry where some parameters are measured on a nearly continuous basis.
Therein we consider $d$ measurements from $N$ subjects which are divided into $a$ independent and generally unbalanced groups where the i-th group contains $n_i$ observations.
Moreover, factor levels on the groups or repeated measures are possible. For independent $d$-dimensional observation vectors $\vX_{ik}\sim\mathcal{N}_d(\mu_i,\Sigma_i)$ null hypotheses  regarding $\vmu=(\vmu_1,...,\vmu_a)^\top$ are investigated, where popular hypotheses are the existence of a group effect, a time effect as well as a interaction effect between time and group. For a classical repeated measure ANOVA design with $d\leq n_i$, this was treated for example in \cite{brunner2012}. But in many cases, it is easier, cheaper or ethically more justifiable to increase the number of repetitions rather than increasing the sample size. 
Therefore techniques are needed, which can handle the case of $d>n_i$.

In the special case with just two groups but with a general distributional setting and without restriction on the dimension $d$ this was treated in \cite{chen2010}.
For more groups and a more general setting regarding hypotheses, \cite{happ2016} uses a 
classical ANOVA F test statistic, which has just an exact F-distribution for very special covariances matrices. So under some conditions on $n_i/d$ or the relation between the dimension and some power of traces containing the covariance matrix, they developed a decent approximation for the test statistic.

In \cite{harrar2016} they handle several cases with an increasing number of groups under some requirements on the covariance matrices and the relation between sample sizes and the number of factor levels.
In contrast, \cite{pEB} investigated the case with just one normal distributed group, but fewer assumptions on the covariance matric and no necessary relation between sample size and dimension.\\

\cite{sattler2018} expand these results especially for a larger number of groups, which is also allowed to approach infinity, together with the sample sizes and the dimension. Hereby, no restrictions on their respective convergence rate were made. 
However, this does not treat the small $n$ large $a$ case which was, e.g., treated by  \cite{bathke2002} or \cite{bathke2005} for fixed dimensions $d$ and balanced designs $n_i\equiv n$.

Therefore in addition to the large $a$ small $n$ case, we include the large $d$ small $n$ case, and further the combination of both, and develop a technique that can be used in each of these settings. To this end, we follow the same approach as \cite{kong2019} and assume homogenous covariance matrices with $\vSigma_i=\vSigma>0$, again with no further assumptions on the structure of the covariance matrix $\vSigma$. The homoscedastic setting allows some generalizations as well as a smaller number of other requirements on the underlying statistical model.\\\\

This paper is organized as follows. Section 2 introduces the statistical model, the investigated hypotheses and the notations used in the remaining paper. In Section 3 the test statistic is presented, as well as their asymptotic behavior and an alternative small sample approximation.
Section 4 contains simulations regarding the type-I-error rate and the power of the tests, introduced in the previous chapters.
The paper closes with a short conclusion. For brevity and readability, all proofs are shifted to the appendix.
\section{Statistical Model and  Hypotheses} \label{mod}

We consider a  homogenous split-plot design given by $a$ independent and unbalanced groups of $d$-dimensional random vectors
\bqan{\vX}_{i,j}= ({X}_{i,j,1},\dots,{X}_{i,j,d})^\top \stackrel{ind}{\sim}\mathcal{N}_d\left(\vmu_i,\vSigma\right)\hspace{0,2cm}j=1,\dots,n_i,\hspace{0,2cm} i=1,\dots,a ,
\eqan
whereby each vector represents the measurement of one independent subject.  It is assumed  that mean vectors $E(\vX_{i,1})=\vmu_i = (\mu_{i,t})_{t=1}^d \in \R^d$ and one positive definite covariance 
matrix $Cov (\vX_{i,1})=\vSigma>0$ exist.
As usual $j=1,\dots,n_i$ denotes the individual subjects or units in group $i=1,\dots, a$, $a,n_i\in \N$, so we have a total number of $N=\sum_{i=1}^a n_i$ random vectors.
This framework allows a factorial structure regarding time, group or both, by splitting up the indices, accordingly, see \cite{kon:2015} for example.\\\\

Within this model  linear hypotheses of repeated measures ANOVA, 
formulated as
\bqan\label{eq:null hypo}
\mathcal H_0(\vH) :\vH\vmu=\vnull \quad   \vmu = (\vmu_1^\top,\ldots,\vmu_a^\top)^\top,
\eqan 
are investigated.
Here, $\vH=\vH_W\otimes \vH_S$ denots a proper hypothesis matrix , where $\vH_W$ and $\vH_S$ refer to  {\bf w}hole-plot (group) 
and/or {\bf s}ubplot (time)  effects, while $\otimes$ denotes the Kronecker product.

 For theoretical considerations it is often more convenient to reformulate $\mathcal H_0(\vH)$ through a corresponding projection matrix
{$\vT =\vH^\top [\vH \vH^\top]^- \vH$}, see e.g. \cite{pEB}. 
Here $(\cdot)^-$ denotes some generalized inverse of the matrix and $\mathcal H_0(\vH)$ can equivalently be written as 
$\mathcal H_0(\vT):\vT\vmu=\vnull$. 
As discussed in \cite{sattler2018}, $\vT$ has the form $\vT=\vT_W\otimes\vT_S$ for projection matrices $\vT_W$ and $\vT_S$. Now hypotheses of interest are for example given by

\bit
\item[(a)] No group effect:\\ $\mathcal H_0^a: \left(\vP_a\otimes \frac{1}{d}\vJ_d\right)\vmu=\vnull$,
\item[(b)] No time effect:\\ $\mathcal H_0^b: \left(\frac 1 a\vJ_a\otimes \vP_d\right)\vmu=\vnull$,
\item[(c)] No interaction effect between time and group:\\
$\mathcal H_0^{ab}: \left(\vP_a\otimes \vP_d\right)\vmu=\vnull$.
\eit

Here, $\vJ_d$ is the d-dimensional matrix only containing 1s and $\vP_d:=\vI_d-1/d\cdot \vJ_d$ is the centring matrix. \\\\

It is often useful to split the expectation vector into its components to simplify the interpretation.
With the common conditions $\sum_i \alpha_i = \sum_t \beta_t = \sum_{i,t} (\alpha\beta)_{it} = 0$, this can be done by expanding $$
\mu_{i,t}= \mu + \alpha_i + \beta_t + (\alpha\beta)_{it},\quad i=1,\dots,a;\ t=1,\dots,d.
$$
Here, $\alpha_i\in\R$ describes the $i$-th group effect, $\beta_t\in \R$ the time effect at time point $t$ and 
$(\alpha\beta)_{it}\in \R$ the $(i,t)$-interaction effect between group and time. Thereby the above hypotheses can alternatively  be formulated through
\bit\item[(a)] $H_0^a: \alpha_i\equiv 0 \text{ for all } i$, 
\item[(b)] $H_0^b: \beta_t\equiv 0 \text{ for all } t,$ \item[(c)]  $H_0^{ab}: (\alpha\beta)_{it}\equiv 0 \text{ for all } i,t$.
\eit

\section{Test statistic and their asymptotic}

In this work we  consider the following 5 different asymptotic frameworks, which are:

\renewcommand{\theequation}{\Roman{equation}}

\setcounter{equation}{0}
\begin{eqnarray}
a\to \infty,\label{asframe1}\\
a,d\to \infty,\label{asframe2}\\
a,n_{\max}\to \infty,\label{asframe3}\\
d,n_{\max}\to \infty,\label{asframe4}\\
a,d,n_{\max} \to \infty\label{asframe5}.
\end{eqnarray}

\renewcommand{\theequation}{\arabic{equation}}
\setcounter{equation}{2}
This great diversity is exceptional and distinguishes from nearly all other approaches. Most of the existing procedures  just consider special cases of one of these cases (for example \cite{chen2010} (IV) with $a=1$ or  \cite{pEB} (IV) with $a=2$).  Other allow for only one as \cite{happ2016} for (IV) or  \cite{bathke2002} for (I).\\\\

In contrast, our framework allows the combination of any of these assumptions. However, $d\to\infty$ alone is not included as this would not allow the construction of consistent trace estimators of covariances which are later needed for inference. Moreover, the case $n_{\max}=\max(n_1,...,n_a)\to \infty$ with fixed $a$ and $d$ has already been studied in detail in the literature and is, thus excluded here, see, e.g., \cite{friedrich2017permuting} or \cite{bathke2018} and the references cited therein.\\\\

\textbf{It is apparent that in contrast to \citet{sattler2018} and other papers, the common conditions as  $\frac{n_i}{N}\to \kappa_i \in (0,1)$ are missing}. This is significant, because it allows an appreciably larger amount of settings, especially for $a\to \infty$. But it also clearly generalizes the model for the case of fixed $a$, e.g. in unbalanced settings, where we only let some group sample sizes converge to $\infty$. \\\\

To examine the validity of the nullhypothesis $H_0(\vT):\vT\vmu=\vnull$ unattached from the asymptotic framework, we use
$Q_N=N\cdot\overline {\vX}^\top \vT\overline {\vX}.$
 Here ${\overline{\vX}=(\overline\vX_1^\top,\dots\overline\vX_a^\top)^\top}$ with $\overline {\vX}_{i} = n_i^{-1} \sum_{j=1}^{n_i} \vX_{i,j}, i=1,\dots,a,$ denotes the vector of pooled group means. 
Unfortunately for many covariance matrices $\vSigma$, this random variable tends to converge to infinity, for $d\to \infty$ or $a\to \infty$. To avoid this behaviour 
 the standardized quadratic form is given by

\[\widetilde W_N=\frac{Q_N-\E_{\mathcal H_0}(Q_N)}{\sqrt{\Var_{\mathcal H_0}(Q_N)}},\]
is used, which also enables us to evaluate all limit distributions in detail.\\\\
For normal distributed observations the expectation and variance of the quadratic form is known and it follows that\\

$\begin{array}{ll}\E(Q_N)&=\tr(\vT_S\vSigma)\cdot \sum\limits_{i=1}^a \frac{N}{n_i} (\vT_W)_{ii}\\[1.4ex]
\Var(Q_N)&=2\Cdot \tr((\vT_S\vSigma)^2)\cdot \sum\limits_{i=1}^a\sum\limits_{r=1}^a \frac{N^2}{n_i n_r} (\vT_W)_{ir}^2.\end{array}$\\\\
Observe, that for both values only the first factor $\tr(\vT_S\vSigma)$ resp.  $\tr((\vT_S\vSigma)^2)$ depends on  the unknown covariance matrix, while all other quantities are known from the test setting.

Applying the representation theorem for quadratic forms in normaly distributed random vectors from\cite{mathaiProvost} we can rewrite the standardized statistic $\widetilde W_N$ as
{\begin{equation}\label{eq: tileW}
\widetilde W_N  = \frac{Q_N-\E_{H_0}(Q_N)}{\Var_{H_0}\left(Q_N\right)^{1/2}} 
\ \stackrel{\mathcal{D}}{=} \ \sum\limits_{s=1}^{ad} 
\frac{\lambda_s}{\sqrt{\sum_{\ell=1}^{ad} \lambda_\ell^2}}\left(\frac{C_s-1}{\sqrt{2}}\right).
\end{equation}
Here $\lambda_s$ are the eigenvalues of $\vT\vV_N\vT$ in decreasing order, $\vV_N=\bigoplus_{i=1}^a \frac{N}{n_i}\vSigma$ and $(C_s)_s$ is a sequence of independent 
$\chi_1^2$-distributed random variables. As a consequence, the asymptotic  behaviour of the eigenvalues, determine the asymptotic limit distribution of $\widetilde W_N$. In fact, we obtain in generalization of \cite{pEB} and \cite{sattler2018}:

\begin{theorem}\label{Asymptotik}
{Let $\beta_s={\lambda_s}\Big/{\sqrt{\sum_{\ell=1}^{ad}\lambda_\ell^2}}$ for $s=1,\dots,ad$. }
Then $\widetilde W_N$ has, under $H_0(\vT)$, and one of the 
frameworks  \eqref{asframe1}-\eqref{asframe5}  asymptotically
\begin{itemize} 
\item[a)]a distribution of the form $\sum_{s=1}^rb_s \left(C_s-1\right)/\sqrt 2+\sqrt{1-\sum_{s=1}^r b_s^2}\cdot Z$,  if and only if
\[\text{for all } s\in \N \hspace{0.5cm}\beta_s\to b_s \hspace{0.5cm} \text{as }\hspace{0.2cm} N \to \infty,\]
for a decreasing sequence $(b_s)_s$ in $[0,1]$ with  $r:=\#\{b_i\neq 0\}$, while $C_i\stackrel{i.i.d.}{\sim} \chi_1^2$, $Z\sim \mathcal{N}(0,1)$.
\item[b)]a distribution of the form $\sum_{s=1}^\infty b_s \left(C_s-1\right)/\sqrt 2$,  if 
\[\text{for all } s\in \N \hspace{0.5cm}\beta_s\to b_s \hspace{0.5cm} \text{as }\hspace{0.2cm} N \to \infty,\]
for a decreasing sequence $(b_s)_s$ in $(0,1)$ with $\sum_{s=1}^\infty b_s^2=1$ and $C_i\stackrel{i.i.d.}{\sim} \chi_1^2$.

\end{itemize}
\vspace{-.5cm}

\end{theorem}
\textbf{Putting the results into context.} \cite{chen2010} only considered case a) with $r=0$. \cite{sattler2018} at least found asymptotic results in case b) but for case a) they need $\boldsymbol{r\in\{0,1\}}$.
So this theorem is not only distinct from other results through the variety of asymptotic settings. It also enhances the continuum of limit distributions considerably through a mixture of normal distribution and finite sums of weighted standardized $\chi_1^2$-distributed random variables. Furthermore, the if and only if relation shows the importance of the demands for the standardized eigenvalues and that it isn't possible to relax them.\\\\


To use this test statistic it is necessary to construct proper estimators which are ratio consistent in all our settings. To this end,  define\\

$A_1=\frac 1 { \sum_{i=1}^a (n_i-1)n_i}\sum\limits_{i=1}^a  \sum\limits_{\ell_1< \ell_2=1}^{n_i} (\vX_{i,\ell_1}-\vX_{i,\ell_2})^\top\vT_S (\vX_{i,\ell_1}-\vX_{i,\ell_2})$\\\\ and
\\\\$A_2=\sum\limits_{i=1}^a \sum\limits_{\begin{footnotesize}\substack{\ell_1,\ell_2=1\\ \ell_1>\ell_2}\end{footnotesize}}^{n_i}
\sum\limits_{\begin{footnotesize}\substack{k_2=1\\k_2\neq \ell_1\neq \ell_2 }\end{footnotesize}}^{n_i}\sum\limits_{\begin{footnotesize}\substack{k_1=1\\ \ell_2\neq \ell_1\neq k_1>k_2}\end{footnotesize}}^{n_i}
\frac{\left[\left({\vX}_{i,\ell_1}-{\vX}_{i,\ell_2}\right)^\top \vT_S\left({\vX}_{i,k_1}-{\vX}_{i,k_2}\right)\right]^2} {4\cdot 6 \sum_{i=1}^a \binom{n_i} {4}}.$\\\\\\
Below we prove that they are unbiased and ratio consistent estimators for $\tr(\vT_S\vSigma)$ and $\tr\left(\left(\vT_S\vSigma\right)^2\right)$, respectivly, under both, the nullhypothesis and the alternative.
This allows us to define the estimated version of our test statistic by
\[W_N=\frac{Q_N-A_1\cdot \sum_{i=1}^a \frac{N}{n_i} (\vT_W)_{ii}}{\sqrt{2\Cdot A_2\cdot \sum_{i=1}^a\sum_{r=1}^a \frac{N^2}{n_i n_r} (\vT_W)_{ir}^2}}.\]

The following Lemma justifies the usage of the estimated version instead of the exact one.
\begin{theorem}\label{Theorem4}
Under $H_0(T):\vT\vmu=\vnull_{ad}$ and one of the 
frameworks \eqref{asframe1}-\eqref{asframe5} the statistic $W_N$ has the same asymptotic limit distributions as $\widetilde W_N$, if the respective conditions (a)-(b) from Theorem~\ref{Asymptotik} are fulfilled.
\vspace{-.5cm}
\end{theorem}

Unfortunately the calculation of the standardized eigenvalues $\beta _s$ is in generally not simplified through homogeneity. Therefore it is nearly impossible to find an appropriate estimator which can be used in all our frameworks. Moreover simulations showed that large sample sizes dimension or number of groups are necessary for a good approximation, which make quantiles based on Theorem~\ref{Asymptotik} a) difficult to apply. For similar reasons, in \cite{pEB} and \cite{sattler2018}  they used the quantils of a random variable of the kind 
\bqan\label{eq:Kf}K_f=(\chi_f^2-f)/\sqrt{2f},\eqan in case of $\beta_1\to \{0,1\}$.  The choice of $f_P=\tr^3\left((\vT\vV_N)^2\right)/\tr^2\left((\vT\vV_N)^3\right)$ for  the degrees of freedom lead to a third moment approximation. In our homoscedastic model the usage of this random variable $K_f$ is based on the following theorem.
\begin{theorem}\label{Theorem5}
Under the conditions of Theorem~\ref{Asymptotik} and one of the 
frameworks \eqref{asframe1}-\eqref{asframe5}  the random variable $K_{{f_P}}$ has, under $H_0:\vT\vmu=\boldsymbol{0}_{ad}$, asymptotically 
\begin{itemize} 
\item[a)]a standard normal distribution  if $\beta_1\to 0$ as $N \to \infty$,
\item[b)]a standardized $\left(\chi_1^2-1\right)/\sqrt{2}$ distribution if $\beta_1\to 1$ as $N \to \infty$.
\end{itemize}
\vspace{-.5cm}
\end{theorem}
With the well known rules for the kronecker product and traces we can decompose this number by

\[f_P=\frac{\tr^3\left(\left(\vT_S\vSigma\right)^2\right)}
{\tr^2\left(\left(\vT_S\vSigma\right)^3\right)}\cdot 
\frac{\tr^3\left( \left[\diag(N/n_1,...,N/n_a)\cdot \vT_W \right]^2\right)}{\tr^2\left( \left[\diag(N/n_1,...,N/n_a)\cdot \vT_W \right]^3\right)}
=: \frac{\tr^3\left(\left(\vT_S\vSigma\right)^2\right)}
{\tr^2\left(\left(\vT_S\vSigma\right)^3\right)}\cdot
{\eta_{N,a}}.\]\\\\

The connection between $f_P$ and $\beta_1$ in the two extreme cases, i.e. $\beta_1\to 0$ if and only $f_P\to \infty$ and $\beta_1\to 1$ if and only if $f_P\to 1$, have been investigated in \cite{pEB} for the case of $a=1$ but also translate to the present framework.\\

Here we have to estimate the first part, while the second one  $\eta_{N,a}$ just depends on the asymptotic setting and therefore is known. This allows us to use the same estimated traces for different hypothesis which  differ only in $\vT_W$.\\
Moreover, for  $\eta_{N,a}\to \infty$, we also have $f_P\to \infty$, without estimation, because ${\tr^3\left(\left(\vT_S\vSigma\right)^2\right)}/
{\tr^2\left(\left(\vT_S\vSigma\right)^3\right)}\geq 1$. Otherwise, however, the behaviour of $f_P$ is unclear and we have to find consistent estimators for $\tr\left(\left(\vT_S\vSigma\right)^3\right)$ in all our different frameworks. This achieved by considering the class of estimators \[C_{i,1}:= \frac{1}{8}\sum\limits_{\ell_1\neq ...\neq \ell_6=1}^{n_i} \vY_{i,\ell_1,\ell_2}^\top \vY_{i,\ell_3,\ell_4}\vY_{i,\ell_3,\ell_4}^\top \vY_{i,\ell_5,\ell_6}\vY_{i,\ell_5,\ell_6}^\top \vY_{i,\ell_1,\ell_2}, \quad \vY_{i,\ell_1,\ell_2}:=\vT_S(\vX_{i,\ell_1}-\vX_{i,\ell_2}), \]
which are based on suitable symmetrized U-statistics, while $\ell_1\neq \ell_2\neq ... \neq \ell_6$ means that all indices are different.\\
 Afterwards these estimators for each individual group  are combined, to get an estimator which uses the observations of each group, given by

\[C_1:=\frac{1}{6! \cdot \sum_{j=1}^a \binom{n_j}{6}}\sum\limits_{i=1}^aC_{i,1}.\]

Together with the estimators from above, we can construct a consistent estimator for $f_P $ by $\widehat f_P:={A_2^3}/{C_1^2}\cdot \eta_{N,a}$.

\begin{theorem}\label{Schaetzerf}
In all our frameworks  \eqref{asframe1}-\eqref{asframe5},  it holds that
\begin{itemize}
\item[i)]
$C_{1}$ is an unbiased  estimator for $\tr\left(\left(\vT_S\vSigma\right)^3\right)$,
\item[ii)]
$\left(\widehat f_P\right)^{-1}-\left(f_P\right)^{-1} \stackrel{\mathcal P}{\to} 0$,

\end{itemize}
where $\mathcal P$ denotes convergence in propabilty.
\end{theorem}

Through the usage of U-statistics with a kernel of order 6,  for each estimator $C_{1} $, $6!\cdot\binom {n_i}{6}$ summations have to be done. In contrast, estimators based on observations from all groups would require much higher numbers. For example in \cite{sattler2018} $\prod_{i=1}^a 6!\cdot \binom {n_i}{6}$ summations are necessary. Due to homogeneity, we don't need this kind of estimator, but  $C_{1,i}$ also requires  $6!\cdot \sum_{j=1}^a \binom {n_j}{6}$ summations, which is already really high, even for comparatively small samples sizes or numbers of groups. Thus, as in \cite{sattler2018}, the usage of subsampling versions of our estimators is reasonable to make them applicable in practice. The underlying idea is instead of summing up all possible index combinations of one group we just do this for a randomly chosen subset of combinations.  \\
}

To define the subsampling version, it is first necessary to introduce some definitions and notations.
A parameter $\upsilon\in (0,\infty)$ is chosen and used to define $w_i=\big\lceil \upsilon\Cdot \binom{n_i}{6}\big\rceil, i=1,...,a$ as the number of subsampling repetitions  done for the i-th group. It is clear that the choice of $\upsilon$ has a great influence on the calculation time and accuracy,  so it should be chosen suitable for the situation.

Then,   random subsamples $\vsigma_i(b) =\{\sigma_{1i}(b),\dots,\sigma_{6i}(b)\}$ of length $6$ from $\{1,\dots,n_i\}$ are drawn independently for each $i=1,\dots,a$ and $b=1,\dots,w_i$, to define the subsampling version of $C_{i,1}$   by
\[C_{i,1}^\star=C_{i,1}^\star(w_i)=\sum\limits_{b=1}^{w_i} \Lambda_1(\vsigma_i(b))\cdot\Lambda_2(\vsigma_i(b))\Cdot \Lambda_3(\vsigma_i(b)). \]

Here \[\Lambda_1(\ell_1,\ell_2,\ell_3,\ell_4,\ell_5,\ell_6)=\vY_{i,\ell_1,\ell_2}^\top\vY_{i,\ell_3,\ell_4},\]
\[\Lambda_2(\ell_1,\ell_2,\ell_3,\ell_4,\ell_5,\ell_6)=\vY_{i,\ell_3,\ell_4}^\top\vY_{i,\ell_5,\ell_6},\]
\[\Lambda_3(\ell_1,\ell_2,\ell_3,\ell_4,\ell_5,\ell_6)=\vY_{i,\ell_5,\ell_6}^\top \vY_{i,\ell_1,\ell_2}.\]\\

Combining them, allows to define the subsamling version of $C_1$ by 
\[C_1^\star:=\frac{1}{8\Cdot  \sum_{j=1}^a w_i}\Cdot\sum\limits_{i=1}^a C_{i,1}^\star(w_i).\]
\begin{theorem}\label{Schaetzerfstar}

For $\sum_{i=1}^a w_i\to \infty$, if $N\to \infty$ (which includes frameworks  \eqref{asframe1}-\eqref{asframe5}) it holds:
\begin{itemize}
\item[a)]
 $C_1^\star$ is  unbiased for $\tr\left(\left(\vT_S\vSigma\right)^3\right)$ .
\item[b)] $\widehat f_P^\star:=\frac{{A_2^3}}{(C_1^\star)^2}\cdot \eta_{N,a}\ $ fullfilles $\ \left(\widehat f_P^\star\right)^{-1}-\left(f_P\right)^{-1}\stackrel{\mathcal P}{\to}0.$
\end{itemize}
\end{theorem}

This way  of defining the number of subsampling repetitions
 $w_i$,  guarantees that the relation between the  subsampled parts $C_{1,i}^\star$ resembles the relation between the original $C_{1,i}$. Although this can lead to great differences between the subsampling sizes for the different groups, it ensures that the influence of single groups is not too big. \\\\

These results allow to formulate a more useable version of $K_{f_P}$ through the following theorem.
\begin{theorem}
The results of  Theorem~\ref{Theorem5} remains valid if $f_P$ is replaced by $\widehat f_P$ or $\widehat f_P^\star $.
\vspace{-.5cm}
\end{theorem}

For the estimation of the unknown traces, it would also be possible to construct estimators that use observations from different groups. This is feasible and seems to be reasonable but in practice, we would again need subsampling versions of these estimators, which take care of the structure of the dataset. This is really complicated and therefore not usable in practice. So we avoid these difficulties by using estimators for the separate groups and combine them afterward.\\\\
{Relaxing the assumption of homogeneous covariance matrices to $\vX_{ij}\sim \mathcal{N}_d(\vmu_i,\vSigma_i)$ with $\vT_S\vSigma_1= \vT_S\vSigma_2=...= \vT_S\vSigma_a$, which is essentially easier to fulfill, wouldn't change the  validity of the previous results. From a theoretical point of view it would be even sufficient to assume $\tr\left(\left(\vT_S\vSigma_1\right)^j\right)= \tr\left(\left(\vT_S\vSigma_2\right)^j\right)=...= \tr\left(\left(\vT_S\vSigma_a\right)^j\right)$ for $ j\in\{1,2,3\}$, but this is nearly impossible to justify in practice.

\section{Simulation}
 For an evaluation of the finite sample behavior of  the introduced method, we have conducted extensive simulations regarding
\begin{itemize}
\item[(i)] their ability in keeping the nominal significance level and
\item[(ii)]their power to detect certain alternatives in various scenarios.\\
\end{itemize}

Here we focus on the frameworks (\ref{asframe1}) and (\ref{asframe2}), which are the most unique ones, because they don't require the usual condition of increasing sample sizes, and therefore they are a strict expansion of the settings considered in \cite{sattler2018}.

\subsection{Type-I error}
To check the type-I error rate for $\alpha=5\%$ we consider  small($d=5$, $d=50$), moderate($d=200$) and large dimension($d=600$)  and  increasing the number of groups from $2$ to $12$. The sample sizes are fixed in a quite unbalanced setting given by $\vn=(n_1,...,n_{12})=(15,15,20,35,25,20,30,30,35,20,15,25)$. We used 10.000 simulation runs and chose $\upsilon=0.05$ for our subsampling type estimators. Thereby, the number of subsamling draws are between 251 and 81.158, one basis of the quite unbalanced setting. Higher values for $\upsilon$ would increase the accurancy but noticeable extend the computation time.

Two different nullhypotheses are investigated, to have a situation with $\beta_1\to 0$ as well as with $\beta_1\to 1$ .
These hypotheses are
\begin{itemize}
\item $\mathcal{H}_0^a: \left(\vP_a\otimes \vP_d\right)\vmu=\vnull$,
\item $\mathcal{H}_0^b: \left(\frac 1 a\vJ_a\otimes\frac 1 d \vJ_d\right)\vmu=\vnull$.\\

\end{itemize}
For both hypotheses the same distributional setting is choosen, with $\vSigma$ as a autoregressive covariance matrix with parameter $0.6$ e.g. $(\vSigma)_{i,j}=0.6^{|i-j|}$ and $\vmu_i=\vnull_d$ for $i=1,...,a$, to achieve better comparabilty. For $\mathcal H_0^b$ it holds $\tau_P\equiv 1$ while the values for $\mathcal{H}_0^a$ can be seen in Table~\ref{Tabelle1}

\begin{table}[ht]
\centering
\begin{tabular}{|l|r|r|r|r|r|r|r|r|r|r|r|r|}
  \hline
 $\tau_P$ & a=2 & a=3 & a=4 & a=5 & a=6 & a=7 & a=8 & a=9 & a=10 & a=11 & a=12 \\ 
  \hline
  d=5& .524 & .268 & .189 & .146 & .122 & .105 & .097 & .092 & .080 & .074 & .070 \\ 
  d=50  & .100 & .051 & .036 & .028 & .023 & .020 & .019 & .018 & .015 & .014 & .013 \\ 
  d=200 & .025 & .013 & .009 & .007 & .006 & .005 & .005 & .004 & .004 & .004 & .003 \\ 
  d=600  & .008 & .004 & .003 & .002 & .002 & .002 & .002 & .001 & .001 & .001 & .001 \\ 
   \hline
\end{tabular}\caption{ $\tau_P$ for $\vT=\frac 1 a\vJ_a\otimes \frac 1 d \vJ_d$ and $(\vSigma)_{ij}=0.6^{|j-i|}$ with different dimension and numbers of groups.}\label{Tabelle1}
\end{table}


\begin{figure}[H]
\setlength{\abovecaptionskip}{5pt} 
\setlength{\belowcaptionskip}{1pt} 
\begin{minipage}[t]{0.49\textwidth}\vspace{0pt} 
\includegraphics[width=\textwidth]{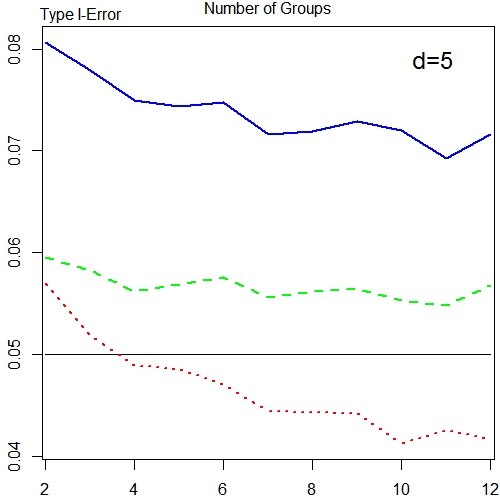} 
\end{minipage} \hspace{0.003cm}\begin{minipage}[t]{0.49\textwidth}\vspace{0pt} 
\includegraphics[width=\textwidth]{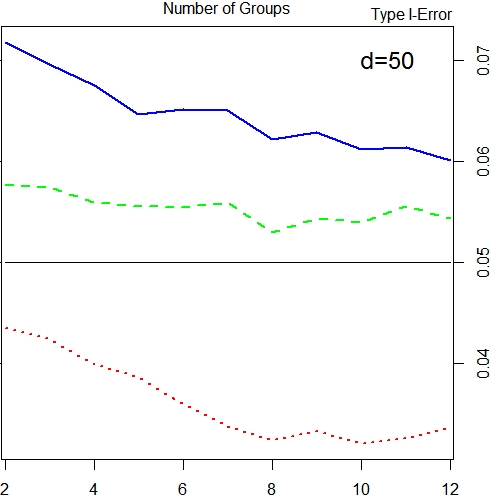} 
\vspace{-1.0cm}\end{minipage} \\
\begin{minipage}[t]{0.49\textwidth}\vspace{0pt} 
\includegraphics[width=\textwidth]{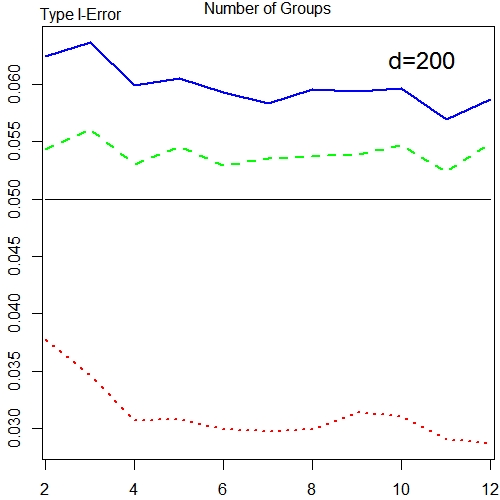} 
\end{minipage} 
\begin{minipage}[t]{0.49\textwidth}\vspace{0pt} 
\includegraphics[width=\textwidth]{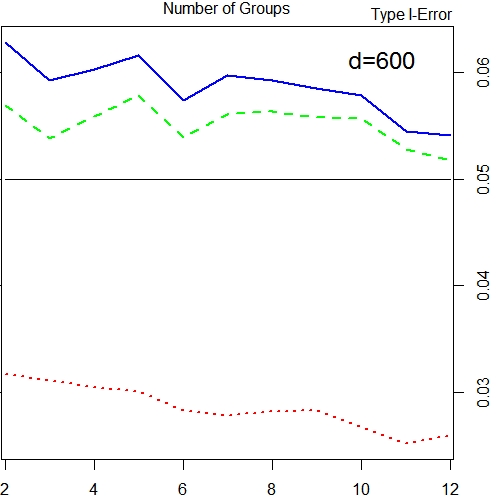} 
\end{minipage}%
\caption{Simulated Type I-Error rates ($\alpha=5\%$) for $\psi_z$(\textcolor{blue}{---}), $\psi_\chi$ (\textcolor{red}{$\cdot \cdot \cdot$}) and
$\varphi_{N}^\star$(\textcolor{green}{- -}) under the null hypothesis $H_0^a: \left(\vP_a\otimes \vP_d\right)\vmu=\vnull$ for increasing dimension. } 
\label{BildSimulation1} 
\end{figure} 
 All tests $\psi_z=\ind(W_N>z_{1-\alpha})$, $\psi_{\chi}=\ind(W_N>\chi_{1;1-\alpha}^2)$ and $\varphi_{N}^\star=\ind\{W_N>K_{\hat f_P;1-\alpha}\}$ are used while $\chi_{1;1-\alpha}^2$ denotes the $1-\alpha$ quantile of a $\chi_1^2$ distribution and $K_{\hat f_P;1-\alpha}$ the $1-\alpha$ quantile of $K_{\hat f_P}$. It must be noted that in the following figures, we use different axes for each setting to make them as detailed as possible.

\begin{figure}[H]
\setlength{\abovecaptionskip}{5pt} 
\setlength{\belowcaptionskip}{1pt} 
\begin{minipage}[t]{0.49\textwidth}\vspace{0pt} 
\includegraphics[width=\textwidth]{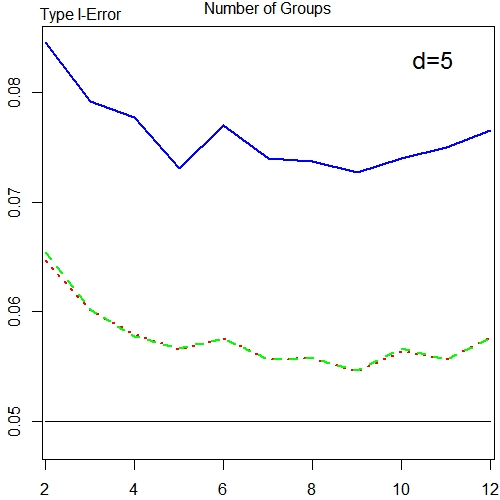} 
\end{minipage} \hspace{0.003cm}\begin{minipage}[t]{0.49\textwidth}\vspace{0pt} 
\includegraphics[width=\textwidth]{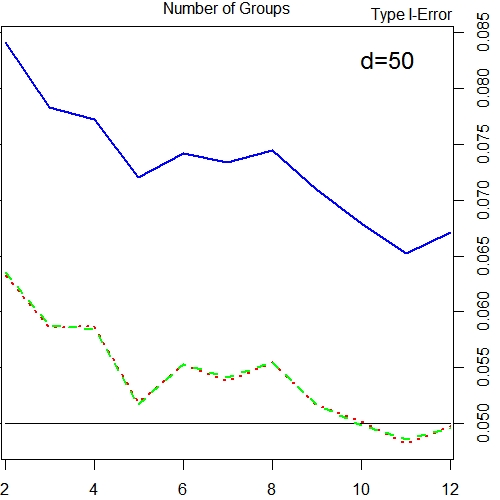} 
\vspace{-1.0cm}\end{minipage} \\
\begin{minipage}[t]{0.49\textwidth}\vspace{0pt} 
\includegraphics[width=\textwidth]{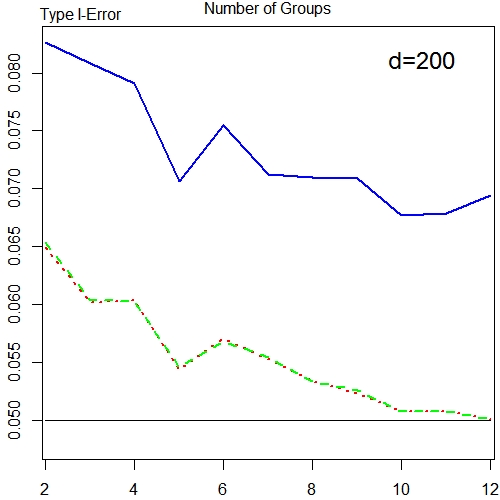} 
\end{minipage} 
\begin{minipage}[t]{0.49\textwidth}\vspace{0pt} 
\includegraphics[width=\textwidth]{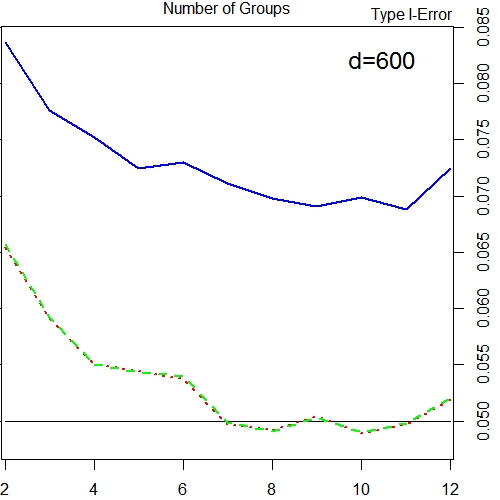} 
\end{minipage}%
\caption{Simulated Type I-Error rates ($\alpha=5\%$) for $\psi_z$(\textcolor{blue}{---}), $\psi_\chi$ (\textcolor{red}{$\cdot \cdot \cdot$}) and
$\varphi_{N}^\star$(\textcolor{green}{- -}) under the null hypothesis $H_0^b: \left(\frac{1}{a}\vJ_a\otimes \frac{1}{d}\vJ_d\right)\vmu=\vnull$ for increasing dimension. } 
\label{BildSimulation2} 
\end{figure}

In Figure~\ref{BildSimulation1} it can be seen that for $\beta_1\to 0$, the usage of $\psi_\chi$ results in too conservative test decisions, especially for larger dimension. So, in this case, a rate which is in most cases lower than 0.04 would lead to a raised number of rejections when the null-hypothesis is true. However, $\psi_z$ has too high type-I error rates, especially in the case of small d=5. But, this improves for a higher dimension as well as a larger number of groups.
For all dimensions $\varphi_N^\star$ shows by far the best type I error control rates and performs well with comparatively low dimensions or just a few groups.
It can be seen that the error rates have less fluctuation for higher numbers of groups. The reason for this is that for fixed comparatively small sample sizes an increasing number of groups not only improves the approximation but also is necessary to get reliable estimators.\\\\

In contrast, there is nearly no difference between $\psi_\chi$ and $\varphi_N^\star$ in Figure~\ref{BildSimulation2} . This is not surprising, because from Figure~\ref{Tabelle1} we know that $f_P$ has always the value one. Furthermore, the small difference between both curves shows once more the good performance of the used estimators. Apart from that again the performance of $\varphi_N^\star$ is quite good, in particular for a higher number of groups. Using the test $\phi_z$ that is based on the wrong limit distribution under $\mathcal{H}_0^b$ results in considerably larger type-I error rates between $0.065$ and $0.085$.\\\\

To sum up $\varphi_N^\star$ shows really good type-I error rates, overall settings, dimensions, and group numbers, even for substantially unbalanced sample sizes, which moreover contains groups with just a few observations.

\subsubsection{Power}
The property to detect deviations from the nullhypothesis is investigated    by considering the same distributional setting as for the type I-error rate, with the same  hypotheses. For this analysis we choose $d=50$ and  small($a=2$), moderate($a=4$) and large($a=8$, $a=10$) number of factor levels.\\\\

We are interested in three kinds of alternatives: 
\begin{itemize}

\item a  \emph{trend-alternative} with $\vmu_1=\vmu_3=....,\vmu_9=\vnull_d$ and $(\vmu_2)_k=(\vmu_4)_k,...,(\vmu_{10})_k=\delta \Cdot k/d$, $k=1,...,d$,\quad $\delta\in [0,2]$,
\item a  \emph{one-point-alternative} with $\vmu_1=\vmu_3=....,\vmu_9=\vnull_d$ and $\vmu_2=\vmu_4,...,\vmu_{10}=\delta \Cdot \ve_1$,\quad $\delta\in [0,3.5]$ and
\item a \emph{shift-alternative} with $\vmu_1=\vmu_3=....,\vmu_9=\vnull$ and $(\vmu_2)=(\vmu_4),...,(\vmu_{10})=\delta\cdot \veins_d$ for $\mathcal{H}_0^b,$\quad $\delta\in [0,2]$
\end{itemize}
Here $\ve_\ell$ denotes the vector containing $1$ in the $\ell-th$ component, and $0$ elsewhere and $\veins_d$ contains just 1's in each component.

\captionsetup[subfloat]{labelformat=empty}
\captionsetup[subfloat]{justification=RaggedRight}
\captionsetup[subfloat]{nearskip=0.3cm}
\captionsetup[subfloat]{farskip=0.2cm}
\captionsetup[subfloat]{margin=15pt}

\begin{figure}[H]
\begin{minipage}[t]{0.99\textwidth}\vspace{0pt} 
\includegraphics[width=\textwidth]{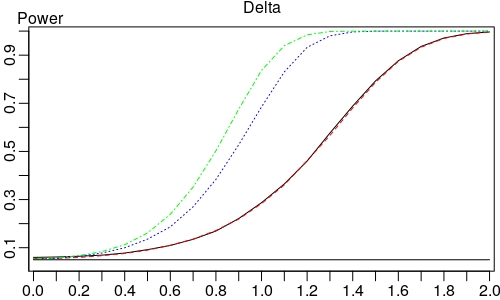} 
\end{minipage} \\
\begin{minipage}[t]{0.99\textwidth}\vspace{0pt} 
\includegraphics[width=\textwidth]{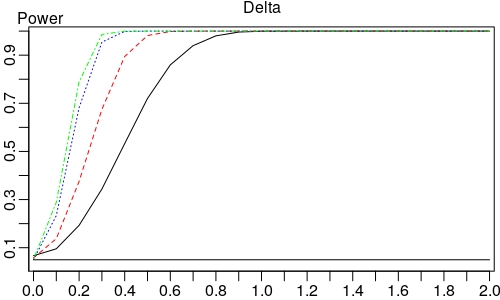} 
\end{minipage} %
\caption{ 
Simulated power curves of $\varphi_N^\star$ for a trend alternative with $d=50$, 10000 simulation runs 
and an autoregressive structure($\left(\vSigma\right)_{i,j}=0.6^{|i-j|}$). The sample size is   $\vn=(15,15,20,35,25,20,30,30,35,20)$ and different numbers of groups were considered, namely $a=2$(---), $a=4$(\textcolor{red}{- -}), $a=8$(\textcolor{blue}{$\cdot \cdot \cdot$}) and $a=10$(\textcolor{green}{$\cdot - \Cdot - $}).  } 
\label{BildPowerT} 
\end{figure} 
From the simulation result given in \cite{sattler2018}, it directly follows that it is challenging to detect the one-point alternative for $d=50$ depending on the hypothesis. For this reason, we here consider a much larger value for $\delta$.

\begin{figure}[htp]
\begin{minipage}[t]{0.49\textwidth}\vspace{0pt} 
\includegraphics[width=\textwidth]{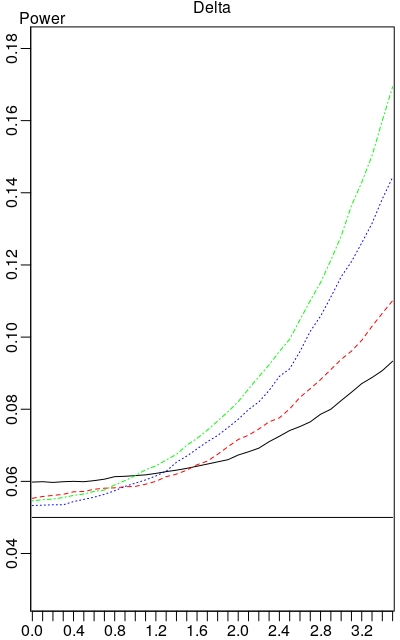} 
\end{minipage} 
\begin{minipage}[t]{0.49\textwidth}\vspace{0pt} 
\includegraphics[width=\textwidth]{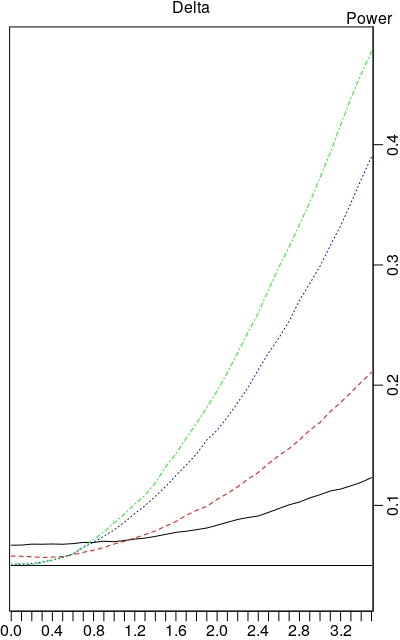} 
\end{minipage} 
\caption{ 
Simulated power curves of $\varphi_N^\star$ for a one-point alternative with $d=50$, 10000 simulation runs 
and an autoregressive structure($\left(\vSigma\right)_{i,j}=0.6^{|i-j|}$). The sample size is   $\vn=(15,15,20,35,25,20,30,30,35,20)$and different numbers of groups were considered, namely $a=2$(---), $a=4$(\textcolor{red}{- -}), $a=8$(\textcolor{blue}{$\cdot \cdot \cdot$}) and $a=10$(\textcolor{green}{$\cdot - \Cdot - $}).  } 
\label{BildPowerTOP} 
\end{figure}

\begin{figure}[h]
\begin{minipage}[t]{0.99\textwidth}\vspace{0pt} 
\includegraphics[width=\textwidth]{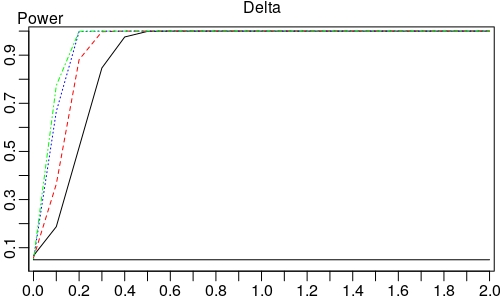} 
\end{minipage} 
 
\caption{ 
Simulated power curves of $\varphi_N^\star$ for a shift alternative with $d=50$, 10000 simulation runs 
and an autoregressive structure( $\left(\vSigma\right)_{i,j}=0.6^{|i-j|}$ ). The sample size is   $\vn=(15,15,20,35,25,20,30,30,35,20)$and different numbers of groups were considered, namely $a=2$(---), $a=4$(\textcolor{red}{- -}), $a=8$(\textcolor{blue}{$\cdot \cdot \cdot$}) and $a=10$(\textcolor{green}{$\cdot - \Cdot - $}).  } 
\label{BildPowerS} 
\end{figure}

%
%
For the trend alternative(Figure~\ref{BildPowerT})  $\varphi_N^\star$ has a high power for both nullhypotheses were the power is essential higher for $\mathcal{H}_0^b$. Increasing the number of groups also increases the power in both hypothesis.  It is noticeable that for $\mathcal H_0^a$ increasing the number from $8$ to $10$ groups has substantial more effect  than from $2$ to $4$ groups while for $\mathcal H_0^b$ it's vice versa.\\\\

As expected detecting the one-point alternative(Figure~\ref{BildPowerTOP}) is challenging for both hypotheses, so the power is low in both cases, even for larger $\delta$- values in particular for $\mathcal{H}_0^a$. This coincides with the power calculations from \cite{sattler2018}. But it can be seen that an increasing number of groups increase the power essentially.

Finally, we considered a shift alternative(Figure~\ref{BildPowerS}), but just for $\mathcal{H}_0^b$. As in other cases(\cite{pEB},\cite{sattler2018}), this alternative is comparatively easy to detect. This holds in particular for an increasing number groups.\\

All in all, except for the one-point alternative, $\varphi_N^\star$ has very high power even for these small sample sizes, especially $n_1=n_2=15$. Moreover $H_0^b$ is much easier to detect, in all settings.

\section{Conclusion}
 The present paper investigated a procedure for homoscedastic split-plot designs under various settings containing different kinds of potential high-dimensionality. Under equal covariance matrices or similar conditions (as mentioned in Section~\ref{mod}) results for settings with, for example, a large number of small independent groups are found. These kinds of data sets nowadays get more important because there is a trend to divide data sets more, e.g. in the context of personalized medicine. Different to existing approaches we take this development into account by considering a variety of different frameworks.\\
We were able to expand the central theorem of \cite{sattler2018} to also cover this case, for the price of the additional assumption of equal covariance matrices. Moreover, we generalized it to some more cases that kind of completes the theorem. For all settings, we approximate the critical value of the test statistic, by a standardized $\chi_f^2$ distribution with appropriate $f$. To use these results we developed estimators that can be used unattached of the asymptotic framework.\\
We conducted simulations to investigate the level of the resulting test as well as its power. The outcomes were convincing, especially for a larger number of groups.\\\\

Unfortunately, it is not that easy to verify the assumption of equal covariance matrices or just equal powers of traces. The most popular test under normality, Box's M-test \cite{box1953}, has quite good results but doesn't take care of our asymptotic frameworks. High-dimensional tests of equal covariance matrices are a field of great interest, which was for example investigated in \cite{li2012} and \cite{li2014} and we plan to combine their techniques with the results obtained in \cite{sattler2019} in the near future.
\section*{Acknowledgement}
The author would like to thank Markus Pauly for helpful discussions and many valuable suggestions. This work was supported by the German Research Foundation project DFG-PA2409/4-1.

\newpage
\section{Appendix}\label{Appendix}
\begin{proof}[Proof of Theorem~\ref{Asymptotik}] For this proof, it is helpful to present the theorem in a more detailed way.\\\\

{Let $\beta_s={\lambda_s}\Big/{\sqrt{\sum_{\ell=1}^{ad}\lambda_\ell^2}}$ for $s=1,\dots,ad$. }
Then $\widetilde W_N$ has, under $H_0(\vT)$, and one of the 
frameworks \ref{asframe1}-\ref{asframe5} asymptotically
\begin{itemize} 
\item[a)]a standard normal distribution  if and only if
\[\beta_1 = \max_{s\leq ad} \beta_{s} \to 0 \hspace{0.5cm} \text{as}\hspace{0.2cm} N \to \infty,\]
\item[b)]a standardized $\left(\chi_1^2-1\right)/\sqrt{2}$ distribution if and only if
\[\beta_1\to 1 \hspace{0.5cm} \text{as}\hspace{0.2cm} N \to \infty,\]

\item[c)]a distribution of the shape  $\sum_{s=1}^rb_s \left(C_s-1\right)/\sqrt 2+\sqrt{1-\sum_{s=1}^r b_s^2}\cdot Z$,  if and only if
\[\text{for all } s\in \N \hspace{0.5cm}\beta_s\to b_s \hspace{0.5cm} \text{as }\hspace{0.2cm} N \to \infty,\]
for a decreasing sequence $(b_s)_s$ in $[0,1)$ with  $r\in \N\setminus \{1\}$ with $b_{r}>0$ and $b_{r+1}=0$ with $C_i\stackrel{i.i.d.}{\sim} \chi_1^2$, $Z\sim \mathcal{N}(0,1)$.
\item[d)]a distribution of the shape  $\sum_{s=1}^\infty b_s \left(C_s-1\right)/\sqrt 2$,  if 
\[\text{for all } s\in \N \hspace{0.5cm}\beta_s\to b_s \hspace{0.5cm} \text{as }\hspace{0.2cm} N \to \infty,\]
for a decreasing sequence $(b_s)_s$ in $(0,1)$ with $\sum_{s=1}^\infty b_s^2=1$ and $C_i\stackrel{i.i.d.}{\sim} \chi_1^2$.

\end{itemize}
The first two parts as well as the last one were proved in \cite{sattler2018}.\\
For part c)  from Cramers theorem it is well known that  it needs an infinite number of summands to get a normal distribution as limit distribution. So it exists a infinite amount $M\subset \N$ with
\[\sum\limits_{\ell\in M} \beta_\ell \left(\frac{C_\ell-1}{\sqrt{2}}\right)\stackrel{\mathcal{D}}{\to}\sqrt{1-\sum_{s=1}^r b_s^2}\cdot Z.\]
The proof of part a) shows, that  $\beta_\ell\to 0$ for all $\ell\in M$, and because of the decreasing order there exists an $r'\in \N$ with $\beta_{r'}>0$ and $\beta_{r'+1}=0$.
Assume now that $\beta_\ell\to b'_\ell$ for $\ell=1,...,r'$ otherwise consider the subsequence where this holds.
It remains to show that from
\[\sum\limits_{\ell=1}^{r'} \beta_\ell \left(\frac{C_\ell-1}{\sqrt{2}}\right)\to \sum\limits_{\ell=1}^{r'} b_\ell' \left(\frac{C_\ell-1}{\sqrt{2}}\right) \stackrel{\mathcal{D}}{=}\sum\limits_{\ell=1}^{r} b_\ell \left(\frac{C_\ell-1}{\sqrt{2}}\right),\]
it follows $r=r'$ as well as $b_\ell=b_\ell'$.
To this aim, we consider the Moment-generating functions so we know, for all $t\in \R$

\[\begin{array}{l}\prod\limits_{\ell=1}^{r'} \left(1-  \frac{2 b_\ell' t} {\sqrt 2} \right)^{-1/2}\exp\left(-  t \frac{b_\ell'} {\sqrt 2}\right)=\prod\limits_{\ell=1}^{r} \left(1-  \frac{2 b_\ell t} {\sqrt 2} \right)^{-1/2}\exp\left(-  t \frac{b_\ell} {\sqrt 2}\right). \end{array}\]\\\\
Thus, applying the continous mapping theorem we have for all $ t \in \R$ \\\\
$\begin{array}{rl}\left(\prod_{\ell=1}^{r'} \left(1-  \frac{2 b_\ell' t} {\sqrt 2} \right)^{-1/2}\exp\left(-  \frac{b_\ell' t} {\sqrt 2}\right)\right)^{-2}&=\left(\prod_{\ell=1}^{r} \left(1-  \frac{2 b_\ell t} {\sqrt 2} \right)^{-1/2}\exp\left(-  \frac{b_\ell t} {\sqrt 2}\right)\right)^{-2} \\[1.5ex]\Leftrightarrow\prod_{\ell=1}^{r'} \left(1-   {\sqrt{2}b_\ell' t} \right) \exp(-\sqrt{2} b_\ell't)&=\prod_{\ell=1}^{r} \left(1-   {\sqrt{2}b_\ell t} \right)\cdot \exp(-\sqrt{2} b_\ell t).\end{array}$\\\\\\
Now we consider the zero points of both sides which are a consequence of the polynomial parts and can be written by $\frac{1}{\sqrt{2}b_\ell}$ resp. $\frac{1}{\sqrt{2}b_\ell'}$. It can be directly inferred from this that both polyiomials has the same degree and therefore $r'=r$. Moreover both of them have the same zero points with the same multiplicity. So the coefficients are the same on both sides and because of the decreasing order  it follows $b_\ell=b_\ell'$ for $\ell=1,...,r$.
Therefore the result follows.
\end{proof}

Given the fact that framework \ref{asframe3} is not really high-dimensional, and \ref{asframe1} just partwise, it would be possible to use other more classical estimators for the unknown traces. Nevertheless our focus was to develop preferably general estimators which can be used in a variety of settings.
\begin{Le}
With 
\[A_{i,1}=\frac 1 2\sum\limits_{\ell_1\neq \ell_2=1}^{n_i} (\vX_{i,\ell_1}-\vX_{i,\ell_2})^\top\vT_S (\vX_{i,\ell_1}-\vX_{i,\ell_2})\]
we can define
\[A_1=\frac 1 { \sum_{i=1}^a (n_i-1)n_i}\sum\limits_{i=1}^a A_{i,1},\]
which is an unbiased and ratio consistent estimator for $\tr(\vT_S\vSigma)$, in all of our frameworks.\end{Le}
\begin{proof}

It is obvious that this is a unbiased estimator of $\tr(\vT_S\vSigma)$. With well known rules and analogous to \cite{sattler2018} we calculate
\[\Var(A_1)\leq\frac 1 {\left[\sum_{i=1}^a \binom {n_i}{2}\right]^2} \sum\limits_{i=1}^a  \binom {n_{i}} {2} \left(\binom {n_{i}} {2} -\binom {n_{i}-2} {2}\right)\cdot\0(\tr^2(\vT_S\vSigma)).\]\\\\
Now we need a case analysis which is done for some of the following proofs. So the first one is in detail and the other proofs are shorter. 
At first we consider the case where $ n_{\max}\to \infty$. Then
\[\begin{array}{ll}\Var(A_1)&\leq \frac 1 {\left[\sum_{i=1}^a \binom {n_i}{2}\right]\cdot \binom{n_{\max}}{2}} \sum\limits_{i=1}^a  \binom {n_{i}} {2} \left(\binom {n_{i}} {2} -\binom {n_{i}-2} {2}\right)\cdot\0(\tr^2(\vT_S\vSigma))
\\[1.5ex]
&\leq\frac 1  {\left[\sum_{i=1}^a \binom {n_i}{2}\right]\cdot \binom{n_{\max}}{2}} \sum\limits_{i=1}^a  \binom {n_{i}} {2} \left(\binom {n_{\max}} {2} -\binom {n_{\max}-2} {2}\right)\cdot\0(\tr^2(\vT_S\vSigma))\\[1.5ex]
&= \frac{\left(\binom {n_{\max}} {2} -\binom {n_{\max}-2} {2}\right)}{\binom{n_{\max}}{2}} \cdot\0(\tr^2(\vT_S\vSigma))\\
&=\0\left( {n_{\max}^{-1}}\right)\cdot \0(\tr^2(\vT_S\vSigma)).
\end{array}
\]
For the other case $n_{\max}$ is bound and $a\to \infty$. In this situation it holds
\[\begin{array}{ll}\Var(A_1)&\leq \frac 1 {\left[\sum_{i=1}^a \binom {n_i}{2}\right]\cdot a \cdot \binom{n_{\min}}{2}} \sum\limits_{i=1}^a  \binom {n_{i}} {2} \left(\binom {n_{\max}} {2} -\binom {n_{\max}-2} {2}\right)\cdot\0(\tr^2(\vT_S\vSigma))
\\
 &=\frac  {\left(\binom {n_{\max}} {2} -\binom {n_{\max-2}} {2}\right)} {a \cdot \binom{n_{\min}}{2}}\cdot\0(\tr^2(\vT_S\vSigma))
\\ [1.5ex]
&=\0\left( {a}^{-1}\right)\cdot \0(\tr^2(\vT_S\vSigma))
\end{array}\]

So dividing by $\tr^2(\vT_S\vSigma)$ and then use the Tschebyscheff inequality leads to the results in both cases.
\end{proof}
For the estimated version of the standardized quadratic form, one more estimator is needed.
\begin{Le}
The estimator given by 
\[A_2=\sum\limits_{i=1}^a \sum\limits_{\begin{footnotesize}\substack{\ell_1,\ell_2=1\\ \ell_1>\ell_2}\end{footnotesize}}^{n_i}
\sum\limits_{\begin{footnotesize}\substack{k_2=1\\k_2\neq \ell_1\neq \ell_2 }\end{footnotesize}}^{n_i}\sum\limits_{\begin{footnotesize}\substack{k_1=1\\ \ell_2\neq \ell_1\neq k_1>k_2}\end{footnotesize}}^{n_i}
\frac{\left[\left({\vX}_{i,\ell_1}-{\vX}_{i,\ell_2}\right)^\top \vT_S\left({\vX}_{i,k_1}-{\vX}_{i,k_2}\right)\right]^2} {4\cdot 6 \sum_{i=1}^a \binom{n_i} {4}},\]
 is a unbiased and ratio-consistent estimator of $\tr\left(\left(\vT_S\vSigma\right)^2\right)$ in all our asymptotic frameworks.

\end{Le}
\begin{proof}
Again the unbiasedness is clear and we consider the variance.

We calculate
\[\begin{array}{ll}\Var(A_2)&= \sum\limits_{i=1}^a\frac{\Var\left( \sum\limits_{\begin{footnotesize}\substack{\ell_1,\ell_2=1\\ \ell_1>\ell_2}\end{footnotesize}}^{n_i}
\sum\limits_{\begin{footnotesize}\substack{k_2=1\\k_2\neq \ell_1\neq \ell_2 }\end{footnotesize}}^{n_i}\sum\limits_{\begin{footnotesize}\substack{k_1=1\\ \ell_2\neq \ell_1\neq k_1>k_2}\end{footnotesize}}^{n_i}
\left[\left({\vX}_{i,\ell_1}-{\vX}_{i,\ell_2}\right)^\top \vT_S\left({\vX}_{i,k_1}-{\vX}_{i,k_2}\right)\right]^2\right)}{\left[{4\cdot 6 \sum_{i=1}^a \binom{n_i} {4}}\right]^2}\\[2.2ex]
&\leq  \frac  {\sum_{i=1}^a \binom{n_i}{4} \binom {n_i} 4-\binom{n_i-4}{4}} {\left[{4\cdot  \sum_{i=1}^a \binom{n_i} {4}}\right]^2} \0\left(\tr^2\left(\left(\vT_S\vSigma_i\right)^2\right)\right) .

\end{array}\]\\
Similar as before for $ n_{\max} \to \infty$ we get  
\[\Var(A_2)\leq \0\left({n_{\max}}^{-1}\right)\cdot\0\left(\tr^2\left(\left(\vT_S\vSigma_i\right)^2\right)\right)\]
and for $ n_{\max}$ bound and $a \to \infty$

\[\Var(A_2)\leq \0\left({a}^{-1}\right)\cdot\0\left(\tr^2\left(\left(\vT_S\vSigma_i\right)^2\right)\right).\]
Again the result follows by using Tschebyscheff's inequality.
\end{proof}
With these theorems the usage of the estimated standardized quadratic form can be justified.
\begin{proof}[Proof of Theorem~\ref{Theorem4}]
The result follows directly by theorem 3.2 from \cite{sattler2018}.
\end{proof}

For the proof of Theorem~\ref{Schaetzerf} we need to show different properties that combined leads to the result. 

\begin{proof}[Proof of Theorem~\ref{Schaetzerf}]
 
We conduct this proof in several steps:
\begin{itemize}
\item[a)] $\E(C_1)=\tr\left(\left(\vT_S\vSigma\right)^3\right),$
\item[b)] $\Var(C_1)= \frac{\sum_{j=1}^a \binom {n_j}{6}\left(\binom {n_j}{6}-\binom{n_j-6}{6}\right)}{\left(\sum_{i=1}^a \binom {n_i}{6}\right)^2}  \cdot \0\left(\tr^3\left(\left(\vT_S\vSigma\right)\right)\right),$
\item[c)] $\frac{C_1}{\tr^{3/2}\left(\left(\vT_S\vSigma\right)^2\right)}-\frac{\tr\left(\left(\vT_S\vSigma\right)^3\right)}{\tr^{3/2}\left(\left(\vT_S\vSigma\right)^2\right)}\stackrel{\mathcal P}{\to} 0$ in our frameworks \ref{asframe1}-\ref{asframe5},
\item[d)] $\frac{C_1^2}{A_2^4}-(f_P)^{-1}\stackrel{\mathcal{P}}{\to} 0$ in our frameworks \ref{asframe1}-\ref{asframe5}.\\\\
\end{itemize}

The results from \cite{sattler2018} directly yield to
\[\E(C_1)=\tr\left(\left(\vT_S\vSigma\right)^3\right)\] and
\[\Var(C_1)=\sum\limits_{i=1}^a\frac{\Var(C_{i,1})}{6!\cdot \sum\limits_{j=1}^a \binom {n_j}{6}}\leq \frac{\sum_{j=1}^a \binom {n_j}{6}\left(\binom {n_j}{6}-\binom{n_j-6}{6}\right)}{\left(\sum_{i=1}^a \binom {n_i}{6}\right)^2} \cdot \0\left( \tr^3\left(\left(\vT_S\vSigma\right)\right)\right)\]

which prooves a) and b). Together with Tschebychefs inequality this leads to an unbiased ratio consistent estimator for $\tr\left(\left(\vT_S\vSigma\right)^3\right)$.\\\\
For part c) we calculate
\[\E\left(\frac{C_1}{\tr^{3/2}\left(\left(\vT_S\vSigma\right)^2\right)}-\frac{\tr\left(\left(\vT_S\vSigma\right)^3\right)}{\tr^{3/2}\left(\left(\vT_S\vSigma\right)^2\right)}\right)=0\] and
\[\Var\left(\frac{C_1}{\tr^{3/2}\left(\left(\vT_S\vSigma\right)^2\right)}-\frac{\tr\left(\left(\vT_S\vSigma\right)^3\right)}{\tr^{3/2}\left(\left(\vT_S\vSigma\right)^2\right)}\right)\]
\[=\frac{\Var\left(C_1\right)}{\tr^3\left(\left(\vT_S\vSigma\right)^2\right)}\leq 27\cdot 
\frac{\sum_{j=1}^a \binom {n_j}{6}\left(\binom {n_j}{6}-\binom{n_j-6}{6}\right)}{\left(\sum_{i=1}^a \binom {n_i}{6}\right)^2} 
\Cdot \Lan(1)\]\\
Again this number is in $\0(n_{\max}^{-1})$ for $n_{\max}\to \infty$ and in $\0(a^{-1})$ for $a\to \infty$. So in both cases the result follows with the Tschebyscheff-inequality.\\\\

At last the proof of part d) is done using the above results.  A similar proof is part of \cite{sattler2018} but we repeat it for better understanding.

With the last lemma it follows for both cases that\\\\
$\begin{array}{ll}
\frac{C_1^2}{\tr^{3}\left(\left(\vT_S\vSigma\right)^2\right)}-\frac 1 {f_P}&=\left(\frac{C_1}{\tr^{3/2}\left(\left(\vT_S\vSigma\right)^2\right)}\right)^2-\left(\frac{\tr\left(\left(\vT_S\vSigma\right)^3\right)}{\tr^{3/2}\left(\left(\vT_S\vSigma\right)^2\right)}\right)^2
\\[2ex]
&\vspace{.25cm}=\left[\frac{C_1}{\tr^{3/2}\left(\left(\vT_S\vSigma\right)^2\right)}-\frac{\tr\left(\left(\vT_S\vSigma\right)^3\right)}{\tr^{3/2}\left(\left(\vT_S\vSigma\right)^2\right)}\right]\left[\frac{C_1}{\tr^{3/2}\left(\left(\vT_S\vSigma\right)^2\right)}+\frac{\tr\left(\left(\vT_S\vSigma\right)^3\right)}{\tr^{3/2}\left(\left(\vT_S\vSigma\right)^2\right)}\right]
\\&=\lan_P(1)\cdot \left[\frac{C_1}{\tr^{3/2}\left(\left(\vT_S\vSigma\right)^2\right)}-\frac{\tr\left(\left(\vT_S\vSigma\right)^3\right)}{\tr^{3/2}\left(\left(\vT_S\vSigma\right)^2\right)}+2\frac{\tr\left(\left(\vT_S\vSigma\right)^3\right)}{\tr^{3/2}\left(\left(\vT_S\vSigma\right)^2\right)}\right]
=\lan_P(1),
\end{array}$\\\\
were for the last step the trace inquality was used together with Slutzky's theorem.
With the ratio-consistency of $A_2$ it follows $A_2/\tr\left(\left(\vT_S\vSigma\right)\right)\stackrel{\mathcal{P}}{\to} 1$ and because of continous mapping $\tr^3\left(\left(\vT_S\vSigma\right)\right)/A_2^3\stackrel{\mathcal{P}}{\to} 1$. This leads to\\\\

$\begin{array}{ll}
\frac{C_1^2}{A_2^3}-(f_P)^{-1}&=
\frac{\tr^{3}\left(\left(\vT_S\vSigma\right)^2\right)}{A_2^3}\frac{C_1^2}{\tr^{3}\left(\left(\vT_S\vSigma\right)^2\right)}-(f_P)^{-1}\\&

=(1 + \lan_P(1))\cdot \frac{1}{\widehat f_P}-\frac{1}{f_P}\\

&=\frac{1}{\widehat f_P}-\frac{1}{f_P}+ \lan_P(1) \Cdot \frac{1}{\widehat f_P}=\lan_P(1).
\end{array}$ \\\end{proof}
It is obvious that this estimator needs a sufficient big amount of groups with at least 6 observations. Similar for the other estimators, which were introduced earlier. From a theoretical point of view a scenario with $n_{\max} \leq 5$ is part of our model. In practice however, this setting is rarely examined. In this case, it would be possible to define some estimators which combine observations from different groups, which would be much more complicated than our estimators.\\

\begin{proof}[Proof of Theorem~\ref{Schaetzerfstar}]
For this proof, some results of \cite{sattler2018} are used and adapted.
First the expactation value of the estimator, using the notation $w:=\sum_{i=1}^a w_i$:\\
$\begin{array}{ll}\E\left(C_1^\star\right)&=\E\left(\frac{1}{8\cdot \sum_{i=1}^a w_i}\sum\limits_{i=1}^a C_{i,1}^\star\right)\\&=\frac{1}{8\cdot w}\sum\limits_{i=1}^a\E\left( C_{i,1}^\star\right)
\\&=\frac{1}{8 \Cdot w}\sum\limits_{i=1}^a\E\left( \sum\limits_{b=1}^{w_i} \Lambda_1(\vsigma_i(b))\cdot\Lambda_2(\vsigma_i(b))\Cdot \Lambda_3(\vsigma_i(b))\right)
\\&=\frac{1}{8\cdot  w}\sum\limits_{i=1}^a\sum\limits_{b=1}^{w_i} \E\left( \Lambda_1(\vsigma_i(b))\cdot\Lambda_2(\vsigma_i(b))\Cdot \Lambda_3(\vsigma_i(b))\right)
\\&=\frac{1}{8\cdot  w}\sum\limits_{i=1}^a w_i\cdot \E\left( \Lambda_1(1,2,3,4,5,6)\cdot\Lambda_2(1,2,3,4,5,6)\Cdot \Lambda_3(1,2,3,4,5,6)\right)
\\&=\frac{1}{8\cdot  w}\sum\limits_{i=1}^a w_i\cdot 8 \tr\left(\left(\vT_S\vSigma\right)^3\right)
\\&= \tr\left(\left(\vT_S\vSigma\right)^3\right).
\end{array}$\\\\

With Theorem A.9 Theorem A.10 and Theorem A.16 from \cite{sattler2018} for the variance we get \\\\
$\begin{array}{ll}
\Var\left(C_1^\star\right)&=\frac{1}{\left(8\cdot w\right)^2}\sum\limits_{i=1}^a \Var\left(C_{i,1}^\star\right)
\\&\leq \frac{1}{\left(8\cdot w\right)^2}\sum\limits_{i=1}^a w_i^2\cdot\left[0+1-\left(1-\frac{1}{w_i}\right)\cdot \frac{\binom{n_i-6}{6}}{\binom{n_i}{6}}\right].
\end{array}$\\\\\\
Again there the same to cases. If $n_{\max}$ is bound and therefore $\max\limits_{i=1,...,a}(w_i)$ is bound, it follows $a\to \infty$ and hereby\\

$\begin{array}{ll}&\frac{1}{\left(8\cdot w\right)^2}\sum\limits_{i=1}^a w_i^2\cdot\left[0+1-\left(1-\frac{1}{w_i}\right)\cdot \frac{\binom{n_i-6}{6}}{\binom{n_i}{6}}\right]
 \\\leq &\frac{1}{\left(8\cdot w\right)\Cdot a \cdot \min\limits_{i=1,...,a}(w_i)}\cdot  \max\limits_{i=1,...,a}(w_i) \sum\limits_{i=1}^a w_i\cdot 1
 \\[1.8ex]=& \Lan\left(a^{-1}\right)\Cdot \frac{\max\limits_{i=1,...,a}(w_i)}{\min\limits_{i=1,...,a}(w_i)}  \\
 =& \Lan\left(a^{-1}\right)
\end{array}$\\\\

while for $n_{\max} \to \infty$ which implies  $\max\limits_{i=1,...,a}(w_i)\to \infty$ we calculate
first \\\\
$\begin{array}{ll}&w_i^2\cdot\left[0+1-\left(1-\frac{1}{w_i}\right)\cdot \frac{\binom{n_i-6}{6}}{\binom{n_i}{6}}\right]\\[1.55ex]
=&w_i\cdot\left[w_i\cdot \left(1- \frac{\binom{n_i-6}{6}}{\binom{n_i}{6}}\right)+ \frac{\binom{n_i-6}{6}}{\binom{n_i}{6}}\right]
\\[1.55ex]
\leq &w_i\cdot\left[\left(\upsilon \cdot \binom{n_i}{6}+1\right)\left(1- \frac{\binom{n_i-6}{6}}{\binom{n_i}{6}}\right)+ \frac{\binom{n_i-6}{6}}{\binom{n_i}{6}}\right]\\[1.55ex]
= &w_i\cdot\left[\upsilon\left(\binom{n_i}{6}- {\binom{n_i-6}{6}}\right)+ 1\right]
\\\leq &w_i\cdot\left[\upsilon\left(\binom{n_{\min}}{6}- {\binom{n_{\min}-6}{6}}\right)+ 1\right]

\end{array}$\\\\
and therefore \\

$\begin{array}{ll}
\leq &\frac{1}{\left(8\cdot w\right)^2}\sum\limits_{i=1}^a w_i^2\cdot\left[0+1-\left(1-\frac{1}{w_i}\right)\cdot \frac{\binom{n_i-6}{6}}{\binom{n_i}{6}}\right]\\
\leq &\frac{1}{\left(8\cdot w\right)^2}\sum\limits_{i=1}^a w_i\cdot\left[\upsilon\left(\binom{n_{\min}}{6}- {\binom{n_{\min}-6}{6}}\right)+ 1\right]
\\\leq &\frac{1}{\left(64\cdot w\right)\cdot \max\limits_{i=1,...,a}(w_i)}\sum\limits_{i=1}^a w_i\cdot\left[\upsilon\left(\binom{n_{\min}}{6}- {\binom{n_{\min}-6}{6}}\right)+ 1\right]
\\\leq &\frac{1}{\left(64\cdot w\right)\cdot \left(\upsilon\cdot\binom{n_{\max}}{6}-1\right)}\sum\limits_{i=1}^a w_i\cdot\left[\upsilon\left(\binom{n_{\min}}{6}- {\binom{n_{\min}-6}{6}}\right)+ 1\right]
\\[1.8ex]= &\frac{\left[\upsilon\left(\binom{n_{\min}}{6}- {\binom{n_{\min}-6}{6}}\right)+ 1\right]}{64\cdot \left(\upsilon\cdot\binom{n_{\max}}{6}-1\right)}
= \Lan\left(n_{\max}^{-1}\right).
\end{array}$\\\\
From this both values the results follows analogous to the proof of Theorem~\ref{Schaetzerf}.\end{proof}

\newpage
\bibliographystyle{apalike}
\bibliography{Literatur}
\end{document}